\documentclass{amsart}
\usepackage{amsfonts, amsbsy, amsmath, amssymb}

\newtheorem{thm}{Theorem}[section]
\newtheorem{lem}[thm]{Lemma}

\newtheorem{thm-con}[thm]{Theorem-Conjecture}
\numberwithin{equation}{section}

\theoremstyle{definition}

\allowdisplaybreaks

\newcommand{\f}{\Bbb F}

\begin{document}

\title{The M\"obius Function of the Affine Linear Group $\text{AGL}(1,\Bbb F_q)$}

\author[Xiang-dong Hou]{Xiang-dong Hou}
\address{Department of Mathematics and Statistics,
University of South Florida, Tampa, FL 33620}
\email{xhou@usf.edu}

\keywords{affine linear group, finite fields, M\"obius function}

\subjclass[2010]{05E15, 11T99, 20G40}

\begin{abstract}
Let $\text{AGL}(1,\Bbb F_q)$ denote the affine linear group of dimension one over the finite field $\Bbb F_q$. We determine the M\"obius function of the lattice of subgroups of $\text{AGL}(1,\Bbb F_q)$.
\end{abstract}

\maketitle

\section{Introduction}

Let $(X,\le)$ be a partially ordered set (poset) which is locally finite, i.e., for any $x,y\in X$, the closed interval $[x,y]=\{z\in X:x\le z\le y\}$ is a finite set. The M\"obius function of $(X,\le)$ is a function $\mu:X\times X\to\Bbb Z$ such that for all $x,y\in X$,
\[
\sum_{z\in[x,y]}\mu(x,z)=\begin{cases}
1&\text{if}\ x=y,\cr
0&\text{otherwise}.
\end{cases}
\]
The M\"obius function $\mu$ determines the partial order $\le$. When $|X|<\infty$, the M\"obius function allows us to recover a function $f:X\to A$, where $A$ is any abelian group, from its summatory function
\[
f_{\le}(x)=\sum_{z\le x}f(z),\qquad x\in X,
\]
through the M\"obius inversion
\[
f(x)=\sum_{z\le x}\mu(z,x)f_{\le}(z).
\]
For more general background on M\"obius functions, see \cite{Bender-Goldman-AMM-1975, Rota-ZWV-1964}.

Let $G$ be a finite group. The M\"obius function of the lattices of subgroups of $G$, denoted by $\mu_G$, is referred to as the M\"obius function of $G$. For $H<K<G$, we usually write $\mu_G(H,G)=\mu_G(H)$, hence $\mu_G(H,K)=\mu_K(H)$. Since $\mu_G(H,K)$ depends only on $H$ and $K$, we sometimes write $\mu(H,K)$. The M\"obius function $\mu_G$ is an important tool for studying the structure of $G$ and it has a variety of applications \cite{Brown-IM-1975, Brown-Thevenaz-JA-1988, Cameron-Maimani-Omidi-Tayfeh-Rezaie-DM-2006, Cameron-Omidi-Tayfeh-Rezaie-E-JC-2006, Hall-QJM-1936, Hawkes-Isaacs-Ozaydin-RMJM-1989, Thevenaz-JCTA-1987}. P. Hall \cite{Hall-QJM-1936} used $\mu_G$ to compute the Eulerian function $\phi_F(G)$ of $G$ which counts the number of onto homomorphisms from an $n$-generator presentation $F$ to $G$. 

Another application of the M\"obius function $\mu_G$ is in combinatorial designs. Let $G$ act on a finite set $X$ and let $\binom Xk$ denote the set of all $k$-element subsets of $X$. For $H<G$, let $f_k(H)$ be the number of elements of $\binom Xk$ fixed by $H$; $f_k(H)$ is easy to compute once the $G$-orbits of $X$ are known. The sum 
\begin{equation}\label{1.1}
g_k(H)=\sum_{K: H<K<G}\mu(H,K)f_k(K)
\end{equation}
gives the number of $B\in\binom Xk$ such that $G_B=H$, where $G_B=\{\sigma\in G:\sigma(B)=B\}$ is the stabilizer of $B$ in $G$. Assume that the action of $G$ on $X$ is $t$-homogeneous, i.e., the action of $G$ on $\binom Xt$ is transitive. With a given $H<G$, if there exists $B\in\binom Xk$ such that $G_B=H$, then $\{\sigma(B):\sigma\in G\}$ is the set of blocks of a $t-(v,k,\lambda)$ design on $X$ which admits $G$ (modulo the kernel of its action on $X$) as an automorphism group, where $v=|X|$ and 
\[
\lambda=\frac{|G|}{|H|}\cdot\frac{k(k-1)\cdots(k-t+1)}{v(v-1)\cdots(v-t+1)};
\]
see \cite{Beth-Jungnickel-Lenz-1999}. Therefore the determination of the nonvanishingness of the sum \eqref{1.1} for all possible $|H|$ would give all possible parameters of the $t$-designs obtained from the action of $G$ on $X$. Let $\f_q$ denote the finite field with $q$ elements. For the possible parameters of the $3$-designs arising from the actions of $\text{PGL}(2,\f_q)$ and $\text{PSL}(2,\f_q)$ on the projective line $\Bbb P^1(\f_q)$, see \cite{Cameron-Maimani-Omidi-Tayfeh-Rezaie-DM-2006, Cameron-Omidi-Tayfeh-Rezaie-E-JC-2006}. For the possible parameters of the $2$-designs arising from the action of the affine linear group $\text{AGL}(1,\f_q)$ (the group of invertible affine transformations of $\f_q$) on $\f_q$, see \cite{Hou-arXiv1707.02315, Sun-TJM-2010, Sun-JAA-2017}.

It appears that there are only a handful of infinite families of finite groups for which the M\"obius function is known. 

\begin{itemize}
\item 
Finite cyclic groups and finite dihedral groups. See for example \cite[Lemmas~ 18 and 19]{Cameron-Maimani-Omidi-Tayfeh-Rezaie-DM-2006}. 

\medskip

\item
Finite nilpotent groups. The M\"obius function of a finite nilpotent group is the product of the M\"obius functions of its Sylow subgroups. For a finite $p$-group $G$ and $H<G$, 
\begin{equation}\label{1.2}
\mu(H,G)=\begin{cases}
(-1)^\alpha p^{\binom\alpha 2}&\text{if $H\vartriangleleft G$ and $G/H\cong\Bbb Z_p^\alpha$},\cr
0&\text{otherwise};
\end{cases}
\end{equation}
see \cite[\S\S2.7, 2.8]{Hall-QJM-1936}. In particular, the M\"obius functions of all finite abelian groups are known.

\medskip

\item
$\text{AGL}(1,\f_p)$ and $\text{PSL}(2,\f_p)$, where $p$ is a prime. The results are given in \cite[\S\S3.5, 3.9]{Hall-QJM-1936}. The M\"obius function of $\text{PGL}(2,\f_p)$, though not given in \cite{Hall-QJM-1936}, can be obtained in a similar manner. For an arbitrary prime power $q$ and $G=\text{PSL}(2,\f_q)$, the values of $\mu(H,G)$, $H<G$, were determined in \cite{Downs-JLMS-1991}; for $G=\text{PSL}(2,\f_q)$ and $\text{PGL}(2,\f_q)$, the values of $\mu(H,K)$ are computed in \cite{Cameron-Maimani-Omidi-Tayfeh-Rezaie-DM-2006, Cameron-Omidi-Tayfeh-Rezaie-E-JC-2006} for {\em some} pairs of subgroups $(H,K)$. However, for $G=\text{AGL}(1,\f_q), \text{PSL}(2,\f_q), \text{PGL}(2,\f_q)$,  the values $\mu(H,K)$ do not seem to have been determined for {\em all} pairs of subgroups $(H,K)$, $H<K<G$. 

\medskip

\item
For the symmetric group $S_n$, the determination of $\mu(1,S_n)$ is a challenging unsolved question which has generated many interesting results \cite{Monks-Dissertation-2012, Shareshian-JCTA-1997}. 

\end{itemize}

The purpose of the present paper is to determine the M\"obius function of the affine linear group $\text{AGL}(1,\f_q)$, where $q$ is an arbitrary prime power. The approach is based on that of \cite{Hou-arXiv1707.02315}. We use Rota's crosscut theorem for computing the M\"obius function of a finite lattice and the results from \cite{Hou-arXiv1707.02315} on the subgroups of $\text{AGL}(1,\f_q)$; these materials are reviewed in Sections~2 and 3. The M\"obius function of $\text{AGL}(1,\f_q)$ is determined in Section~4. 

\section{Rota's Crosscut Theorem}

Let $(L,\le)$ be a finite lattice with the minimum element $\hat 0$ and the maximum element $\hat 1$. A {\em lower crosscut} of $L$ is a set $A\subset L\setminus\{\hat 0\}$ such that for each $y\in L\setminus(A\cup\{\hat 0\})$, there exists $x\in A$ such that $x<y$. An {\em upper crosscut} of $L$ is a set $B\subset L\setminus\{\hat 1\}$ such that for each $y\in L\setminus(B\cup\{\hat 1\})$, there exists $x\in A$ such that $x>y$.

\begin{thm}[Rota's crosscut theorem {\cite[Theorem~3.1.9]{Kung-Rota-Yan-2009}}]\label{T2.1}
Let $L$ be a finite lattice and let $\mu$ denote its M\"obius function. If $A$ is a lower crosscut of $L$, then
\begin{equation}\label{2.1}
\mu(\hat 0,\hat 1)=\sum_{E:\, E\subset A,\, \bigvee\!E=\hat 1}(-1)^{|E|}.
\end{equation}
If $B$ is an upper crosscut of $L$, then
\begin{equation}\label{2.1}
\mu(\hat 0,\hat 1)=\sum_{E:\, E\subset B,\, \bigwedge\!E=\hat 0}(-1)^{|E|}.
\end{equation}
\end{thm}

For finite groups $H\lneqq G$, consider the lattice of subgroups between $H$ and $G$ (inclusive). In this case, the set of all immediate supergroups of $H$ (the set of minimal elements of $\{K:H\lneqq K<G\}$) is a lower crosscut; the set of all maximal subgroups of $G$ containing $H$ is an upper crosscut.

\section{$\text{AGL}(1,\f_q)$ and Its Subgroups}

We define the affine linear group of dimension 1 over $\f_q$ as 
\begin{equation}\label{3.1}
\text{AGL}(1,\f_q)=\Bigl\{ \left[\begin{matrix} a&b\cr 0&1\end{matrix}\right]:a\in\f_q^*,\ b\in\f_q\Bigr\}<\text{GL}(2,\f_q).
\end{equation}
The group $\text{AGL}(1,\f_q)$ acts on $\f_q$ as follows: For $\left[\begin{smallmatrix} a&b\cr 0&1\end{smallmatrix}\right]\in\text{AGL}(1,\f_q)$ and $x\in\f_q$,
\begin{equation}\label{3.2}
\left[\begin{matrix} a&b\cr 0&1\end{matrix}\right]x=ax+b.
\end{equation}
We have
\begin{equation}\label{3.2.0}
\text{AGL}(1,\f_q)=A\ltimes B,
\end{equation}
where
\begin{equation}\label{3.3}
A=\Bigl\{ \left[\begin{matrix} a&0\cr 0&1\end{matrix}\right]:a\in\f_q^*\Bigr\}\cong\f_q^*,
\end{equation}
\begin{equation}\label{3.4}
B=\Bigl\{ \left[\begin{matrix} 1&b\cr 0&1\end{matrix}\right]:b\in\f_q\Bigr\}\cong\f_q.
\end{equation}
For each $H<\f_q$, define
\begin{equation}\label{3.5}
\overline H=\Bigl\{ \left[\begin{matrix} 1&h\cr 0&1\end{matrix}\right]:h\in H\Bigr\}.
\end{equation}
Let $p=\text{char}\,\f_q$. Let $\mathcal S$ be the set of triples $(a,b,H)$, where $a\in\f_q^*$, $H$ is an $\f_p(a)$-submodule of $\f_q$, and
\[
b
\begin{cases}
=0&\text{if}\ a=1,\cr
\in\f_q&\text{if}\ a\ne 1.
\end{cases}
\]
For $(a,b,H)\in\mathcal S$, define
\begin{equation}\label{3.6}
S(a,b,H)=\Bigl\langle\Bigl\{ \left[\begin{matrix} a&b\cr 0&1\end{matrix}\right]\Bigr\}\cup\overline H\Bigr\rangle<\text{AGL}(1,\f_q),
\end{equation}
where $\langle\ \rangle$ denotes the subgroup generated by a subset. It was proved in \cite{Hou-arXiv1707.02315} that
\begin{equation}\label{3.7.0}
S(a,b,H)=\Bigl\langle\left[\begin{matrix} a&b\cr 0&1\end{matrix}\right]\Bigr\rangle\ltimes\overline H
\end{equation}
and that the multiplicative order of $\left[\begin{smallmatrix} a&b\cr 0&1\end{smallmatrix}\right]$ equals that of $a$; see \cite[\S3.2]{Hou-arXiv1707.02315}.
We fix a generator $\gamma$ of $\f_q^*$.

\begin{thm}[\cite{Hou-arXiv1707.02315}]\label{T3.1}
\begin{itemize}
\item[(i)] Every subgroup of $\text{\rm AGL}(1,\f_q)$ is of the form $S(\gamma^{(q-1)/d},b,H)$, where $d\mid q-1$ and $(\gamma^{(q-1)/d},b,H)\in\mathcal S$. Moreover, $d$ and $H$ are uniquely determined by the subgroup and $b$ is uniquely determined modulo $H$.

\medskip

\item[(ii)] For each $(\gamma^{(q-1)/d},b,H)\in\mathcal S$, where $d\mid q-1$, we have
\begin{equation}\label{3.7}
\left[\begin{matrix} 1&c\cr 0&1\end{matrix}\right]S(\gamma^{(q-1)/d},b,H)\left[\begin{matrix} 1&c\cr 0&1\end{matrix}\right]^{-1}=S(\gamma^{(q-1)/d},0,H),
\end{equation}
where $c=0$ if $d=1$ and $c=b/(\gamma^{(q-1)/d}-1)$ otherwise.

\medskip

\item[(iii)] Let $(\gamma^{(q-1)/d},0,H)\in\mathcal S$, where $d\mid q-1$. If $d\ne 1$, the subgroups of $\text{\rm AGL}(1,\f_q)$ that contain $S(\gamma^{(q-1)/d},0,H)$ are precisely $S(\gamma^{(q-1)/d_1},0,H_1)$, where $d\mid d_1\mid q-1$, $H\subset H_1$, and $(\gamma^{(q-1)/d_1},0,H_1)\in\mathcal S$. If $d=1$, the subgroups of $\text{\rm AGL}(1,\f_q)$ that contain $S(1,0,H)$ are precisely $S(\gamma^{(q-1)/d_1},b_1,H_1)$, where $d_1\mid q-1$, $H\subset H_1$, and $(\gamma^{(q-1)/d_1},b_1,H_1)\in\mathcal S$. 

\end{itemize}
\end{thm}


\section{The M\"obius Function of $\text{\rm AGL}(1,\f_q)$}

Let $S_1<S_2<\text{\rm AGL}(1,\f_q)$. Our goal is to determine $\mu(S_1,S_2)$. By Theorem~\ref{T3.1}, we may assume that
\begin{equation}\label{4.1}
S_i=S(\gamma^{(q-1)/d_i},0,H_i),\qquad i=1,2,
\end{equation}
where $d_1\mid d_2\mid q-1$, $H_1\subset H_2$, and $(\gamma^{(q-1)/d_i},0,H_i)\in\mathcal S$. To justify this assumption, we first assume that $S_1=S(\gamma^{(q-1)/d_1},0,H_1)$ by Theorem~\ref{T3.1} (i) and (ii). If $d_1\ne 1$, by Theorem~\ref{T3.1} (iii), $S_2=S(\gamma^{(q-1)/d_2},0,H_2)$. If $d_1=1$, a suitable conjugation in \eqref{3.7} takes $S_2$ to $S(\gamma^{(q-1)/d_2},0,H_2)$ while leaving $S_1$ unchanged.

For any positive integer $n$, $\mu(n)$ denotes the value of the classic M\"obius function at $n$, i.e.,
\[
\mu(n)=\begin{cases}
(-1)^t&\text{if $n$ is a product of $t$ distinct primes},\cr
0&\text{if $n$ is divisible by the square of a prime}.
\end{cases}
\]
Let $\mu_q$ be the M\"obius function of the lattice of finite dimensional vector spaces over $\f_q$: For finite dimensional $\f_q$-spaces $W\subset V$,
\begin{equation}\label{4.2} 
\mu_q(W,V)=(-1)^lq^{\binom l2},\qquad l=\dim_{\f_q}(V/W).
\end{equation}
We write $\mu_q(0,V)$ as $\mu_q(V)$. Recall that $p=\text{char}\,\f_q$. For any $d\in\Bbb Z$ with $p\nmid d$, let $p(d)$ denote the smallest power of $p$ that is $\equiv 1\pmod d$. 

\begin{thm}\label{T4.1}
Let $S_i=S(\gamma^{(q-1)/d_i},0,H_i)$, $i=1,2$, where $d_1\mid d_2\mid q-1$, $H_1\subset H_2$, and $(\gamma^{(q-1)/d_i},0,H_i)\in\mathcal S$. If $H_1$ is not an $\f_{p(d_2)}$-module, $\mu(S_1,S_2)=0$. If $H_1$ is an $\f_{p(d_2)}$-module,
\begin{equation}\label{4.3}
\mu(S_1,S_2)=\begin{cases}
\mu(d_2/d_1)\,\mu_{p(d_2)}(H_2/H_1)&\text{if}\ d_1\ne 1,\vspace{2mm}\cr
\displaystyle\frac{|H_2|}{|H_1|}\mu(d_2)\,\mu_{p(d_2)}(H_2/H_1)&\text{if}\ d_1=1.
\end{cases}
\end{equation}
\end{thm}

We need some additional notation for the proof of Theorem~\ref{T4.1}. For each positive integer $n$, let $\mathcal P(n)$ be the set of prime divisors of $n$. If $P\subset\Bbb Z$ is finite, define $\Pi P=\prod_{e\in P}e$. For a vector space $V$ over a field $F$, $\Bbb P_F(V)$ denotes the set of all $1$-dimensional $F$-subspaces of $V$. For an $\f_r$-space $W$, let $\mathcal L_r(W)$ be the set of all $\f_r$-subspaces of $W$. While $\langle\ \rangle$ denotes the subgroup generated by a subset, $\langle\ \rangle^K$, where $K$ is a subfield of $\f_q$, denotes the $K$-span of a subset.  

\begin{proof}[Proof of Theorem~\ref{T4.1}]
Let $\phi:\f_q\to\f_q/H_1$ be the canonical homomorphism. Let $H_1'=\{x\in\f_q:xH_1\subset H_1\}$, which is the largest subfield $K$ of $\f_q$ such that $H_1$ is a $K$-module. We consider the cases $d_1\ne 1$ and $d_1=1$ separately.

\medskip
{\bf Case 1.} Assume that $d_1\ne 1$. By Theorem~\ref{T3.1}, the immediate supergroups of $S_1$ are 
\begin{equation}\label{4.4}
S(\gamma^{(q-1)/d_1e},0,H_1),\qquad e\in\mathcal P((|H_1'|-1)/d_1),
\end{equation}
and 
\begin{equation}\label{4.5}
S(\gamma^{(q-1)/d_1},0,\phi^{-1}(x)),\qquad x\in\Bbb P_{\f_{p(d_1)}}(\f_q/H_1).
\end{equation}
For $P\subset\mathcal P((|H_1'|-1)/d_1)$ and $X\subset \Bbb P_{\f_{p(d_1)}}(\f_q/H_1)$, by \cite[Lemma~4.3]{Hou-arXiv1707.02315},
\begin{align}\label{4.6}
&\Bigl\langle S_1\cup\Bigl(\bigcup_{e\in P}S(\gamma^{(q-1)/d_1e},0,H_1)\Bigr)\cup\Bigl(\bigcup_{x\in X}S(\gamma^{(q-1)/d_1},0,\phi^{-1}(x))\Bigr)\Bigr\rangle\\
=\,&S\Bigl(\gamma^{(q-1)/d_1\Pi P},\, 0,\, \Bigl\langle H_1\cup\Bigl(\bigcup_{x\in X}\phi^{-1}(x)\Bigr)\Bigr\rangle^{\f_{p(d_1\Pi P)}}\Bigr). \nonumber
\end{align}
Thus by Theorem~\ref{T2.1},
\begin{equation}\label{4.7}
\mu(S_1,S_2)=\sum_{\substack{P\subset\mathcal P((|H_1'|-1)/d_1),\, X\subset \Bbb P_{\f_{p(d_1)}}(\f_q/H_1)\cr d_1\Pi P=d_2,\, \langle H_1\cup(\bigcup_{x\in X}\phi^{-1}(x))\rangle^{\f_{p(d_2)}}=H_2}} (-1)^{|P|+|X|}.
\end{equation}
In \eqref{4.7}, the condition $d_1\Pi P=d_2$ implies two things: First, $d_2/d_1=\Pi P$ is square-free. Second, $d_2\mid |H_1'|-1$, and hence $H_1$ is an $\f_{p(d_2)}$-module. Therefore $\mu(S_1,S_2)=0$ unless $d_2/d_1$ is square-free and $H_1$ is an $\f_{p(d_2)}$-module. 

Now assume that $d_2/d_1$ is square-free and $H_1$ is an $\f_{p(d_2)}$-module. Then \eqref{4.7} becomes
\begin{align}\label{4.8}
\mu(S_1,S_2)\,&=(-1)^{|\mathcal P(d_2/d_1)|}\sum_{\substack{X\subset \Bbb P_{\f_{p(d_1)}}(\f_q/H_1)\cr \langle H_1\cup(\bigcup_{x\in X}\phi^{-1}(x))\rangle^{\f_{p(d_2)}}=H_2}} (-1)^{|X|}\\
&=\mu(d_2/d_1)\sum_{\substack{L\in\mathcal L_{p(d_1)}(H_2/H_1)\cr \langle L\rangle^{\f_{p(d_2)}}=H_2/H_1}}\ \sum_{\substack{X\subset \Bbb P_{\f_{p(d_1)}}(H_2/H_1)\cr \langle \bigcup_{x\in X}x\rangle^{\f_{p(d_1)}}=L}} (-1)^{|X|}.\nonumber
\end{align}
By Theorem~\ref{T2.1} again, the inner sum in \eqref{4.8} equals $\mu_{p(d_1)}(L)$. Hence 
\[
\mu(S_1,S_2)=\mu(d_2/d_1)\sum_{\substack{L\in\mathcal L_{p(d_1)}(H_2/H_1)\cr \langle L\rangle^{\f_{p(d_2)}}=H_2/H_1}}\mu_{p(d_1)}(L)=\mu(d_2/d_1)\,\mu_{p(d_2)}(H_2/H_1),
\]
where the last step follows from Lemma~\ref{L4.2} at the end of this section.

\medskip
{\bf Case 2.} Assume that $d_1=1$. To avoid a trivial situation, we assume that $d_2\ne 1$. Let $\mathcal B$ be a system of coset representatives of $H_1$ in $\f_q$. By Theorem~\ref{3.1}, the immediate supergroups of $S_1$ are 
\begin{equation}\label{4.9}
S(\gamma^{(q-1)/e},b,H_1),\qquad e\in\mathcal P(|H_1'|-1),\ b\in\mathcal B,
\end{equation}
and 
\begin{equation}\label{4.10}
S(1,0,\phi^{-1}(x)),\qquad x\in\Bbb P_{\f_p}(\f_q/H_1).
\end{equation}
For $C\subset \mathcal P(|H_1'|-1)\times\mathcal B$ and $X\subset\Bbb P_{\f_p}(\f_q/H_1)$, by \cite[Lemma~4.1]{Hou-arXiv1707.02315}, 
\begin{align}\label{4.11}
&\Bigl\langle S_1\cup\Bigl(\bigcup_{(e,b)\in C}S(\gamma^{(q-1)/e},b,H_1)\Bigr)\cup\Bigl(\bigcup_{x\in X}S(1,0,\phi^{-1}(x))\Bigr)\Bigr\rangle\\
=\,&S\Bigl(\gamma^{(q-1)/\Pi P},\, (\gamma^{(q-1)/\Pi P}-1)u,\, \Bigl\langle\Delta(C)\cup H_1\cup\Bigl(\bigcup_{x\in X}\phi^{-1}(x)\Bigr)\Bigr\rangle^{\f_{p(\Pi P)}}\Bigr),\nonumber
\end{align}
where $P\subset\mathcal P(|H_1'|-1)$ is defined by $C=\bigcup_{e\in P}(\{e\}\times B_e)$, $\emptyset\ne B_e\subset\mathcal B$ for all $e\in P$, $u=b/(\gamma^{(q-1)/e}-1)$ for some (any) $(e,b)\in C$, and 
\begin{equation}\label{4.12}
\Delta(C)=\Bigl\{\frac{b_1}{\gamma^{(q-1)/e_1}-1}-\frac{b_2}{\gamma^{(q-1)/e_2}-1}:(e_1,b_1), (e_2,b_2)\in C\Bigr\}.
\end{equation}
Note that both $P$ and $u$ depend on $C$. (If $C=\emptyset$, $\gamma^{(q-1)/\Pi P}-1=0$ and $u$ is irrelevant.) By Theorem~\ref{T2.1}, 
\begin{equation}\label{4.13}
\mu(S_1,S_2)=\sum_{\substack{C\subset \mathcal P(|H_1'|-1)\times\mathcal B,\, X\subset\Bbb P_{\f_p}(\f_q/H_1)\cr \Pi P=d_2,\, (\gamma^{(q-1)/d_2}-1)u\in H_2\cr \langle\Delta(C)\cup H_1\cup(\bigcup_{x\in X}\phi^{-1}(x))\rangle^{\f_{p(d_2)}}=H_2}} (-1)^{|C|+|X|}.
\end{equation}
We assume that $d_2$ is square-free and $H_1$ is an $\f_{p(d_2)}$-module. (Otherwise, $\mu(S_1,S_2)=0$ as we saw in Case 1.) Then \eqref{4.13} becomes 
\begin{equation}\label{4.14}
\mu(S_1,S_2)=\sum_{\substack{C\subset \mathcal P(d_2)\times\mathcal B\cr
P=\mathcal P(d_2),\, u\in H_2}}(-1)^{|C|}\sum_{\substack{X\subset\Bbb P_{\f_p}(\f_q/H_1)\cr \langle\Delta(C)\cup H_1\cup(\bigcup_{x\in X}\phi^{-1}(x))\rangle^{\f_{p(d_2)}}=H_2}}(-1)^{|X|}.
\end{equation}
We claim that if $\Delta(C)\not\subset H_1$, the inner sum in \eqref{4.14} is $0$. In fact, let $\widetilde H_1=\langle\Delta(C)\cup H_1\rangle^{\f_p}$. Then the inner sum equals
\begin{align*}
&\sum_{\substack{X\subset\Bbb P_{\f_p}(\f_q/H_1)\cr \langle \widetilde H_1/H_1\rangle^{\f_{p(d_2)}}+\langle \bigcup_{x\in X}x\rangle^{\f_{p(d_2)}}=H_2/H_1}}(-1)^{|X|}\cr
=\,& \sum_{\substack{X_1\subset\Bbb P_{\f_p}(H_2/H_1)\setminus \Bbb P_{\f_p}(\widetilde H_1/H_1)\cr \langle \widetilde H_1/H_1\rangle^{\f_{p(d_2)}}+\langle \bigcup_{x\in X_1}x\rangle^{\f_{p(d_2)}}=H_2/H_1}}(-1)^{|X_1|} \sum_{X_2\subset \Bbb P_{\f_p}(\widetilde H_1/H_1)}(-1)^{|X_2|}=0.
\end{align*}
Therefore, in \eqref{4.14}, we assume that $\Delta(C)\subset H_1$. This implies that $|B_e|=1$ for each $e\in P=\mathcal P(d_2)$; let $B_e=\{b_e\}$. Then $\Delta(C)\subset H_1$ if and only if there exists $c\in\mathcal B$ such that 
\[
\frac{b_e}{\gamma^{(q-1)/d_2}-1}\equiv c\pmod{H_1}\quad \text{for all}\ e\in\mathcal P(d_2).
\]
The condition $u\in H_2$ is equivalent to $c\in H_2$. Thus, to choose $b_e$, $e\in\mathcal P(d_2)$, such that $\Delta(C)\subset H_1$, we can first choose $c\in\mathcal B\cap H_2$; there are $|H_2/H_1|$ choices for $c$. Then all $b_e$, $e\in\mathcal P(d_2)$, are uniquely determined. Therefore \eqref{4.14} becomes
\begin{align*}
\mu(S_1,S_2)\,&=|H_2/H_1|(-1)^{|\mathcal P(d_2)|}\sum_{\substack{X\subset\Bbb P_{\f_p}(\f_q/H_1)\cr \langle H_1\cup(\bigcup_{x\in X}\phi^{-1}(x))\rangle^{\f_{p(d_2)}}=H_2}}(-1)^{|X|}\cr
&=\frac{|H_2|}{|H_1|}\mu(d_2)\sum_{\substack{L\in\mathcal L_p(H_2/H_1)\cr \langle L\rangle^{\f_{p(d_2)}}=H_2/H_1}}\ \sum_{\substack{X\subset\Bbb P_{\f_p}(H_2/H_1)\cr \langle\bigcup_{x\in X}x)\rangle^{\f_p}=L}}(-1)^{|X|}\cr
&=\frac{|H_2|}{|H_1|}\mu(d_2)\sum_{\substack{L\in\mathcal L_p(H_2/H_1)\cr \langle L\rangle^{\f_{p(d_2)}}=H_2/H_1}}\mu_p(L) \kern1.02cm\text{(by Theorem~\ref{T2.1})}\cr
&=\frac{|H_2|}{|H_1|}\mu(d_2)\,\mu_{p(d_2)}(H_2/H_1) \kern2cm \text{(by Lemma~\ref{L4.2})}.
\end{align*}
\end{proof}

The following lemma has been used in the proof of Theorem~\ref{T4.1}.

\begin{lem}\label{L4.2}
Let $\f_r\subset\f_q$ and let $V$ be a finite dimensional vector space over $\f_q$. Then
\[
\sum_{\substack{L\in\mathcal L_r(V)\cr \langle L\rangle^{\f_q}=V}}\mu_r(L)=\mu_q(V).
\]
\end{lem}

\begin{proof}
For each $U\in\mathcal L_q(V)$, let
\[
\nu(U)=\sum_{\substack{L\in\mathcal L_r(V)\cr \langle L\rangle^{\f_q}=U}}\mu_r(L).
\]
Then
\begin{align*}
\sum_{\substack{W\in\mathcal L_q(V)\cr W\subset U}}\nu(W)\,&=\sum_{\substack{W\in\mathcal L_q(V)\cr W\subset U}}\sum_{\substack{L\in\mathcal L_r(U)\cr \langle L\rangle^{\f_q}=W}}\mu_r(L)=\sum_{L\in\mathcal L_r(U)}\mu_r(L)\cr
&=\begin{cases}
1&\text{if}\ U=\{0\},\cr
0&\text{otherwise}.
\end{cases}
\end{align*}
Hence $\nu=\mu_q$ on $\mathcal L_q(V)$, in particular, $\nu(V)=\mu_q(V)$.
\end{proof}



\end{document}